\documentclass[12pt]{amsart}
\usepackage[latin1]{inputenc}
\usepackage{mathptmx}
\usepackage{amscd}
\usepackage{amssymb}
\usepackage{amsfonts}
\newtheorem{theorem}{Theorem}
\newtheorem{lemma}[theorem]{Lemma}
\newtheorem{corollary}[theorem]{Corollary}

\theoremstyle{definition}
\newtheorem{remark}[theorem]{Remark}
\newtheorem{definition}[theorem]{Definition}

\newtheorem{example}[theorem]{Example}

\begin{document}

\title{Multiplication Semimodules}

\author{Rafieh Razavi Nazari}

\address{Rafieh Razavi Nazari, Faculty of Mathematics, K. N. Toosi University of Technology, Tehran, Iran}
\email{rrazavi@mail.kntu.ac.ir}

\author{Shaban Ghalandarzadeh}

\address{Shaban Ghalandarzadeh, Faculty of Mathematics, K. N. Toosi University of Technology, Tehran, Iran}
\email{ghalandarzadeh@kntu.ac.ir}



\begin{abstract}
Let $S$ be a semiring. An $S$-semimodule $M$ is called a multiplication semimodule if for each subsemimodule $N$ of $M$
there exists an ideal $I$ of $S$ such that $N=IM$. In this paper we investigate some properties of multiplication semimodules and generalize some results on multiplication modules to semimodules. We show that every multiplicatively cancellative multiplication semimodule is finitely generated and projective. Moreover, we characterize finitely generated cancellative multiplication $S$-semimodules when $S$ is a yoked semiring such that every maximal ideal of $S$ is subtractive.
\end{abstract}
\maketitle 

\section{Introduction}
\label{intro}
A semiring is a nonempty set $S$ together with two binary
operations addition $(+)$ and multiplication $(\cdot)$ such that $(S,+)$ is a commutative monoid with identity element $0$; $(S,.)$ is a monoid with identity element $1\neq 0$;  $0a=0=a0$ for all $a\in S$; $a(b+c)=ab+ac$ and $(b+c)a=ba+ca$ for every $a,b,c\in S$. We say that $S$ is a commutative semiring if the monoid $(S,.)$ is commutative. In this paper we assume that all semirings are commutative. A nonempty subset $I$ of a semiring $S$ is called an ideal of $S$ if $a+b\in I$ and $sa\in I$ for all $a,b\in I$ and $s\in S$. A semiring $S$ is called yoked if for all $a, b\in S$, there exists an element $t$ of $S$ such that $a+t=b$ or $b+t=a$. An ideal $I$ of a semiring $S$ is subtractive if $a+b\in I$ and $b\in I$ imply that
$a\in I$ for all $a,b\in S$. A semiring $S$ is local if it has a unique maximal ideal. A semiring is entire if $ab=0$ implies that $a=0$ or $b=0$. An element $s$
of a semiring $S$ is a unit if there exists an element $s^{\prime}$ of $S$ such that $ss^{\prime}=1$. A semiring $S$ is called a semidomain if for any nonzero element $a$ of $S$, $ab=ac$ implies that $b=c$. An element $a$ of a semiring $S$ is called multiplicatively idempotent if $a^2=a$. The semiring $S$ is multiplicatively idempotent if every element of $S$ is multiplicatively idempotent.

Let $(M,+)$ be an additive abelian monoid
with additive identity $0_{M}$. Then $M$ is called an $S$-semimodule if there exists a scalar multiplication $S\times M\rightarrow M$
denoted by $(s,m)\mapsto sm$, such that $(ss^{\prime})m=s(s^{\prime}m)$; $s(m+m^{\prime})= sm+sm^{\prime}$; $(s+s^{\prime})m=sm+s^{\prime}m$; $1m=m$ and $s0_{M}=0_{M}=0m$ for all $s,s^{\prime}\in S$ and all $m,m^{\prime}\in M$.
A subsemimodule $N$ of a semimodule $M$ is a nonempty subset of $M$ such that $m+n\in N$ and $sn\in N$ for all $m,n\in N$ and $s\in S$. 

Let $R$ be a ring. An $R$-module $M$ is a multiplication module if for each submodule $N$ of $M$ there exists an ideal $I$ of $R$ such that $N=IM$\cite{B1}. Multiplication semimodules are defined similarly. These semimodules have been studied by several authors(e.g. \cite{E2}, \cite{E3}, \cite{T2}, \cite{Y}). In this paper, we study multiplication semimodules and extend some results of \cite{E4} and \cite{S2} to semimodules over semirings. In Section \ref{sec:1}, we show that if $M$ is a multiplication $S$-semimodule and $\mathfrak{p}$ is a maximal ideal of $S$ such that $M\neq \mathfrak{p}M$, then $M_{\mathfrak{p}}$ is
cyclic. Moreover we prove that if $M$ is a finitely generated $S$-semimodule and for every maximal ideal $\mathfrak{p}$ of $S$, $M_{\mathfrak{p}}$ is cyclic, then $M$ is a multiplication $S$-semimodule. In Section \ref{sec:2}, we study multiplicatively cancellative(abbreviated as $MC$) multiplication semimodules. We show that $MC$ multiplication semimodules are finitely generated and projective.
In Section \ref{sec:3}, we characterize finitely generated cancellative multiplication semimodules over yoked semirings with subtractive maximal ideals.

\section{Multiplication semimodule}
\label{sec:1}
In this section we give some results of multiplication semimodules which are related to the corresponding results in multiplication modules.

Let $N$ and $L$ be subsemimodules of $M$ and $(N:L)=\{s\in S\mid sL\subseteq N \}$. Then it is clear that $(N:L)$ is an ideal of $S$.
\begin{definition}\cite{E3}
	Let $S$ be a semiring and $M$ an $S$-semimodule.
	Then $M$ is called a multiplication semimodule if for each subsemimodule $N$ of $M$
	there exists an ideal $I$ of $S$ such that $N=IM$. In this case it is easy to prove that $N=(N:M)M$. For
	example, every cyclic $S$-semimodule is a multiplication $S$-semimodule \cite[Example 2]{Y}.
\end{definition}
\begin{example}
	Let $S$ be a multiplicatively idempotent semiring. Then every ideal of $S$ is a multiplication $S$-semimodule. Let $J$ be an ideal of $S$ and $I\subseteq J$. If $x\in I$, then $x=x^2\in IJ$. Therefore $I=IJ$ and hence $J$ is a multiplication $S$-semimodule.
\end{example}

It is well-known that every homomorphic image of a multiplication module is a multiplication module (cf. \cite{E4} and \cite[Note 1.4]{T1}). A similar result holds for multiplication semimodules. 

\begin{theorem}
	Let $S$ be a semiring, $M$ and $N$ $S$-semimodules and $f:M\rightarrow N$ a surjective $S$-homomorphism. If $M$ is a multiplication $S$-semimodule, then $N$ is a multiplication \linebreak$S$-semimodule.
\end{theorem}
\begin{proof}
	Let $N^{\prime}$ be a subsemimodule of $N$ and $M^{\prime}=\{m\in M\mid f(m)\in N^{\prime}\}$. Then $f(M^{\prime})=N^{\prime}$. Since $M$ is a multiplication $S$-semimodule, there exists an ideal $I$ of $S$ such that $M^{\prime}=IM$. Then $N^{\prime}=f(M^{\prime})=f(IM)=If(M)=IN$. Therefore $N$ is a multiplication $S$-semimodule.
\end{proof}

Fractional and invertible ideals of semirings have been studied in \cite{G1}. We recall here some definitions and properties. An element $s$ of a semiring $S$ is multiplicatively cancellable
(abbreviated as $MC$), if $sb=sc$ implies $b=c$ for all $b,c\in S$. We denote the set of all $MC$
elements of $S$ by $MC(S)$. The total quotient semiring of $S$, denoted by $Q(S)$, is defined as the localization of $S$ at $MC(S)$. Then $Q(S)$ is an $S$-semimodule and $S$ can be regarded as a subsemimodule of $Q(S)$. For the concept of the localization in semiring theory, we refer to \cite{K} and \cite{L1}.
A subset $I$ of $Q(S)$ is called a fractional ideal of $S$ if $I$ is a subsemimodule of $Q(S)$ and there exists an $MC$ element $d\in S$ such that $dI \subseteq S$. Note that every ideal of $S$ is a fractional ideal. 
The product of two fractional ideals is defined by $IJ=\{a_1b_1+\ldots+a_nb_n\mid a_i\in I, b_i\in J \}$.
A fractional ideal $I$ of a semiring $S$ is called an invertible ideal if there exists a fractional ideal $J$ of $S$ such that $IJ=S$.

Now we restate the following property of invertible ideals from \cite{G1}(see also \cite[Proposition 6.3]{L3}).

\begin{theorem}
	Let $S$ be a semiring. An ideal $I$ of $S$ is invertible if and only if it is a multiplication $S$-semimodule which contains an
	$MC$ element of $S$.
\end{theorem}

Let $M$ be an $S$-semimodule and $\mathfrak{p}$ a maximal ideal of $S$. Then similar to \cite{E4}, we define 
$T_{\mathfrak{p}}(M)=\{m\in M\mid\ there\ exist\ s\in S\ and\  q\in \mathfrak{p}\ such\ that\ s+q=1\ and\ sm=0 \}$. Clearly $T_{\mathfrak{p}}(M)$ is a subsemimodule of $M$. We say that $M$ is $\mathfrak{p}$-cyclic if there exist $m\in M$, $t\in S$ and  $q\in \mathfrak{p}$ such that $t+q=1$ and  $tM\subseteq Sm$.

The following two theorems can be thought of as a generalization of \cite[Theorem 1.2]{E4}(see also \cite[Proposition 3]{E2}).
\begin{theorem}
	\label{17}
	Let $M$ be an $S$-semimodule. If for every maximal ideal $\mathfrak{p}$ of $S$ either $T_\mathfrak{p}(M)=M$ or $M$ is $\mathfrak{p}$-cyclic, then $M$ is a multiplication semimodule.
\end{theorem}
\begin{proof}
	Let $N$ be a subsemimodule of $M$ and $I=(N:M)$. Then $IM\subseteq N$. Let $x\in N$ and $J=\{s\in S\mid sx\in IM\}$. Clearly $J$ is an ideal of $S$. If $J\neq S$, then by \cite[Proposition 6.59]{G2} there exists a maximal ideal $\mathfrak{p}$ of $S$ such that $J\subseteq \mathfrak{p}$. If $M=T_\mathfrak{p}(M)$, then there exist $s\in S$ and $q\in \mathfrak{p}$ such that $s+q=1$ and $sx=0\in IM$. Hence $s\in J\subseteq \mathfrak{p}$ which is a contradiction. So the second case will happen. Therefore there exist $m\in M$, $t\in S$ and  $q\in \mathfrak{p}$ such that $t+q=1$ and $tM\subseteq Sm$. Thus $tN$ is a subsemimodule of $Sm$ and $tN=Km$ where $K=\{s\in S\mid sm\in tN \}$. Moreover, $tKM=KtM\subseteq Km\subseteq N$. Therefore $tK\subseteq I$. Thus $t^2x\in t^2N=tKm\subseteq IM$. Hence $t^2\in J\subseteq \mathfrak{p}$ which is a contradiction. Therefore $J=S$ and $x\in IM$.    
\end{proof} 
\begin{theorem}
	\label{1}
	Suppose that $M$ is an $S$-semimodule. If $M$ is a multiplication semimodule, then for every maximal ideal $\mathfrak{p}$ of $S$ either $M=\{m\in M\mid m=qm\ for\ some\ q\in \mathfrak{p}\}$ or $M$ is $\mathfrak{p}$-cyclic.
\end{theorem}
\begin{proof}
	Let $\mathfrak{p}$ be a maximal ideal of $S$ and $M=\mathfrak{p}M$. If $m\in M$, then there exists an ideal $I$ of $S$ such that $Sm=IM$. Hence $Sm=I\mathfrak{p}M=\mathfrak{p}IM=\mathfrak{p}m$. Therefore $m=qm$ for some $q\in \mathfrak{p}$. Now let $M\neq \mathfrak{p}M$. Thus there exists $x\in M$ such that $x\notin \mathfrak{p}M$. Then there exists ideal $I$ of $S$ such that $Sx=IM$. If $I\subseteq \mathfrak{p}$, then $x\in IM\subseteq \mathfrak{p}M$ which is a contradiction. Thus $I\nsubseteq \mathfrak{p}$ and since $\mathfrak{p}$ is a maximal ideal of $S$, $\mathfrak{p}+I=S$. Thus there exist $t\in I$ and $q\in \mathfrak{p}$ such that $q+t=1$. Moreover, $tM\subseteq IM=Sx$. Therefore $M$ is $\mathfrak{p}$-cyclic.  
\end{proof}

By using Theorem \ref{1}, we obtain the following corollary.
\begin{corollary}
	\label{13}
	Suppose that $(S,\mathfrak{m})$ is a local semiring. Let $M$ be a multiplication \linebreak$S$-semimodule such that 
	$M\neq \mathfrak{m}M$. Then $M$ is a cyclic semimodule. 
\end{corollary}
\begin{proof}
	Since $M\neq \mathfrak{m}M$, $M$ is $\mathfrak{m}$-cyclic. Thus there exist $n\in M$, $t\in S$ and  $q\in \mathfrak{m}$ such that $t+q=1$ and $tM\subseteq Sn$. Since $S$ is a local semiring, by \cite[Theorem 3.5]{K}, $t$ is unit. Hence $M=Sn$.   
\end{proof}
\begin{remark}
	\label{16}
	Let $S$ be a semiring and $T$ a non-empty multiplicatively closed subset of $S$, and let $M$ be an $S$-semimodule.
	Define a relation $\sim$ on $M\times T$ as follows: 
	$(m,t)\sim(m^{\prime},t^{\prime}) \Longleftrightarrow \exists s\in T$ such that $stm^{\prime}=st^{\prime}m$. The
	relation $\sim$ on $M\times T$ is an equivalence relation.
	Denote the set $M\times T/\sim$ by $T^{-1}M$ and the equivalence class of each
	pair $(m,s)\in M\times T$ by $m/s$. We can define addition on $T^{-1}M$ by $m/t+m^{\prime}/t^{\prime}=(t^{\prime}m+tm^{\prime})/tt^{\prime}$. Then $(T^{-1}M,+)$ is an abelian monoid. Let $s/t\in T^{-1}S$ and $m/u\in T^{-1}M$. We can define the product of $s/t$ and $m/u$ by $(s/t)(m/u)=sm/tu$.
	Then $T^{-1}M$ is an $T^{-1}S$-semimodule\cite{D}.
	Let $\mathfrak{p}$ be a prime ideal in $S$ and $T=S\backslash \mathfrak{p}$. Then $T^{-1}M$ is denoted by $M_{\mathfrak{p}}$.
	
	We can obtain the following results as in \cite{N}.
	\begin{enumerate}
		\item Suppose that $I$ is an ideal of a semiring $S$ and $M$ is an $S$-semimodule. Then 
		$T^{-1}IT^{-1}M=T^{-1}(IM)$.\label{14}

		\item Let $N$, $N^{\prime}$ be subsemimodules of an $S$-semimodule $M$. If $N_{\mathfrak{m}}=N^{\prime}_{\mathfrak{m}}$ for every maximal ideal $\mathfrak{m}$, then $N=N^{\prime}$.\label{15}
		
	\end{enumerate}
\end{remark}

\begin{theorem}
	Let $S$ be a semiring and $M$ a multiplication $S$-semimodule. If $\mathfrak{p}$ is a maximal ideal of $S$ such that $M\neq\mathfrak{p}M$, then $M_{\mathfrak{p}}$ is
	cyclic.
\end{theorem}  
\begin{proof}
	By (\ref{14}), $M_{\mathfrak{p}}$ is a multiplication $S_{\mathfrak{p}}$-semimodule. 
	Since $M\neq \mathfrak{p}M$, $M_{\mathfrak{p}}\neq \mathfrak{p}_{\mathfrak{p}}M_{\mathfrak{p}}$ by (\ref{15}). Moreover, by \cite[Theorem 4.5]{K}, $S_{\mathfrak{p}}$ is a local semiring. Thus by Corollary \ref{13}, $M_{\mathfrak{p}}$ is cyclic.
\end{proof}
In the following result, we extend \cite[Theorem 4.21]{T1} to semirings.
\begin{theorem}
	Let $S$ be a semiring and $M$ a finitely generated $S$-semimodule. If for every maximal ideal $\mathfrak{p}$ of $S$, $S_\mathfrak{p}$-semimodule $M_\mathfrak{p}$ is cyclic, then $M$ is a multiplication \linebreak$S$-semimodule.
	\begin{proof}
		Let $x\in M$ and $t\in S\backslash \mathfrak{p}=T$ such that $M_{\mathfrak{p}}=S_\mathfrak{p}(x/t)$. Let $m\in M$. Then $m/1\in S_\mathfrak{p}(x/t)$. Thus there exist $r\in s$ and $t'\in T$ such that $m/1=(x/t)(r/t')$. Hence there exists $t''\in T$ such that $t''t'tm=t''rx\in (x)$. Moreover $a=t''t't\in T$ and hence $(a)+\mathfrak{p}=S$. Thus $r'a+q=1$ for some $r'\in S$ and $q\in \mathfrak{p}$. Then  $r'am=r't''rx\in (x)$. Now Let $M=\sum_{i=1}^{n} Sm_i$. Then for each $1\leq i\leq n$, there exist $q_i\in \mathfrak{p}$ and $t_i\in S$ such that $t_i+q_i=1$ and $t_im_i\in (x)$. Thus $1=\prod_{i=1}^{n}t_i+b$ for some $b\in \mathfrak{p}$ and $(\prod_{i=1}^{n}t_i)M\subseteq (x)$. Hence $M$ is a multiplication semimodule by Theorem \ref{17}.
	\end{proof}
\end{theorem}
\section{Multiplicatively cancellative multiplication semimodules}
\label{sec:2}
In this section, we study multiplicatively cancellative multiplication semimodules and give some properties of these semimodules.

An $S$-semimodule $M$ is called multiplicatively cancellative(abbreviated as $MC$) if for any $s,s^{\prime}\in S$ and $0\neq m\in M$, $sm=s^{\prime}m$ implies $s=s^{\prime}$\cite{E1}.
For example every ideal of a semidomain $S$ is an $MC$ $S$-semimodule. 

Note that if $M$ is an $MC$ $S$-semimodule, then $M$ is a faithful semimodule. Let $tM=\{0\}$ for some $t\in S$. If $0\neq m\in M$, then $tm=0m=0$. Thus $t=0$. Therefore $M$ is faithful.

Moreover, for an $R$-module $M$ over a domain $R$, $M$ is an $MC$ semimodule if and only if it is torsionfree. Also we know that if $R$ is a domain and $M$ a faithful multiplication $R$-module, then $M$ will be a torsionfree $R$-module and so $M$ is an $MC$ semimodule. 

An element $m$ of an $S$-semimodule $M$ is called cancellable if $m+m_1=m+m_2$ implies that $m_1=m_2$. The semimodule $M$ is cancellative if and only if every element of $M$ is cancellable \cite[P. 172]{G2}.

\begin{lemma}\label{12}
	Let $S$ be a yoked entire semiring and $M$ a cancellative faithful multiplication $S$-semimodule. Then $M$ is an $MC$ semimodule.
\end{lemma}
\begin{proof}
	Let $0\neq m\in M$ and $s,s^{\prime}\in S$ such that $sm=s^{\prime}m$. Since $S$ is a yoked semiring, there exists $t\in S$ such that $s+t=s^{\prime}$ or $s^{\prime}+t=s$. Suppose that $s+t=s^{\prime}$. Then $sm+tm=s^{\prime}m$. Since $M$ is a cancellative $S$-semimodule, $tm=0$. Moreover, there exists an ideal $I$ of $S$ such that $Sm=IM$ since $M$ is a multiplication $S$-semimodule. Then $tIM=tSm=\{0\}$ and hence $tI=\{0\}$ since $M$ is faithful. But $S$ is an entire semiring, so $t=0$. Therefore $s=s^{\prime}$.
	Now suppose that $s^{\prime}+t=s$. A similar argument shows that $s=s^{\prime}$. Therefore $M$ is an $MC$ semimodule.
\end{proof}  

We now give the following definition similar to \cite[P. 127]{L2}.

\begin{definition}
	Let $S$ be a semidomain. An $S$-semimodule $M$ is called a torsionfree semimodule if for any $0\neq a\in S$, multiplication by $a$ on $M$ is injective, i.e., if $ax=ay$ for some $x,y\in M$, then $x=y$.  
\end{definition}
\begin{theorem}
	Let $S$ be a yoked semidomain and $M$ a cancellative torsionfree $S$-semimodule. Then $M$ is an $MC$ semimodule.
\end{theorem}
\begin{proof}
	Let $0\neq m\in M$ and $s,s^{\prime}\in S$ such that $sm=s^{\prime}m$. Since $S$ is a yoked semiring, there exists $t\in S$ such that $s+t=s^{\prime}$ or $s^{\prime}+t=s$. Suppose that $s+t=s^{\prime}$. Then $sm+tm=s^{\prime}m$. Since $M$ is a cancellative $S$-semimodule, $tm=0$. Since $M$ is a torsionfree $S$-semimodule, $m=0$ which is a contradiction. Thus $t=0$ and hence $s=s^{\prime}$. Now suppose that $s^{\prime}+t=s$. A similar argument shows that $s=s^{\prime}$. Therefore $M$ is an $MC$ semimodule.
\end{proof}
Now, similar to \cite[Lemma 2.10]{E4} we give the following theorem (see also \cite[Theorem 3.2]{E3}).
\begin{theorem}
	\label{18}
	Let $\mathfrak{p}$ be a prime ideal of $S$ and $M$ an $MC$ multiplication semimodule. Let $a\in S$ and $x\in M$ such that $ax\in \mathfrak{p}M$. Then $a\in \mathfrak{p}$ or $x\in \mathfrak{p}M$.
\end{theorem}
\begin{proof}
	Let $a\notin \mathfrak{p}$ and 
	put $K=\{s\in S\mid sx\in \mathfrak{p}M \}$. If $K\neq S$, there exists a maximal ideal $\mathfrak{q}$ of $S$ such that $K\subseteq \mathfrak{q}$. By Theorem \ref{1}, $M$ is $\mathfrak{q}$-cyclic since $M$ is an $MC$ semimodule. Therefore there exist $m\in M$, $t\in S$ and $q\in \mathfrak{q}$ such that $t+q=1$ and  $tM\subseteq Sm$. Thus $tx=sm$ for some $s\in S$. Moreover, $t\mathfrak{p}M\subseteq \mathfrak{p}m$. Hence $tax\in t\mathfrak{p}M\subseteq \mathfrak{p}m$. Therefore $tax=p_1m$ for some $p_1\in \mathfrak{p}$ and hence $asm=p_1m$. Since $M$ is an $MC$ semimodule, $as=p_1\in \mathfrak{p}$ and since $\mathfrak{p}$ is a prime ideal, $s\in \mathfrak{p}$ . Then $tx=sm\in \mathfrak{p}M$ and hence $t\in K\subseteq \mathfrak{q}$ which is a contradiction. Thus $K=S$. Therefore $x\in \mathfrak{p}M$. 
\end{proof}
\begin{lemma}(cf. \cite{A})
	\label{2}
	Suppose that $S$ is a semiring. Let $M$ be an $S$-semimodule and $\theta(M)=\sum_{m\in M}(Sm:M)$. If $M$ is a multiplication $S$-semimodule, then $M=\theta(M)M$.
\end{lemma}
\begin{proof}
	Suppose that $m\in M$. Then $Sm=(Sm:M)M$. Thus $m\in (Sm:M)M\subseteq \theta(M)M$. Therefore $M=\theta(M)M$. 
\end{proof}
\begin{theorem}(cf. \cite[Theorem 3.1]{E4})
	\label{3}
	Let $S$ be a semiring and $M$ an $MC$ multiplication $S$-semimodule. Then the following statements
	hold:
	\begin{enumerate} 
		\item If $I$ and $J$ are ideals of $S$ such that $IM\subseteq JM$ then $I\subseteq J$.
		\item $M\neq IM$ for any proper ideal $I$ of $S$.\label{4}
		\item M is finitely generated.
	\end{enumerate} 
\end{theorem}
\begin{proof}
	$(1)$ Let $IM\subseteq JM$ and $a\in I$. Set $K=\{ s\in S\mid sa\in J \}$. If $K\neq S$, there exists a maximal ideal $\mathfrak{p}$ of $S$ such that $K\subseteq \mathfrak{p}$. By Theorem \ref{1}, $M$ is $\mathfrak{p}$-cyclic since $M$ is an $MC$ semimodule. Thus there exist $m\in M$, $t\in S$ and  $q\in \mathfrak{p}$ such that $t+q=1$ and  $tM\subseteq Sm$. Then $tam\in tIM\subseteq tJM=JtM\subseteq Jm$. Hence there exists $b\in J$ such that $tam=bm$. Since $M$ is an $MC$ semimodule, $ta=b\in J$. Thus $t\in K\subseteq \mathfrak{p}$ which is a contradiction. Therefore $K=S$ and hence $I\subseteq J$. 
	
	$(2)$ Follows by $(1)$

	$(3)$ By Lemma \ref{2}, $M=\theta(M)M$, where $\theta(M)=\sum_{m\in M}(Sm:M)$. Then by \ref{4}, $\theta(M)=S$. Thus there exist a positive integer $n$ and elements $m_i\in M$ and $r_i\in (Sm_i:M)$ such
	that $1=r_1+\ldots+r_n$. If $m\in M$, then $m=r_1m+\ldots+r_nm$. Therefore $M=Sm_1+\ldots+Sm_n$. 
\end{proof}
By Lemma \ref{12}, we have the following result.
\begin{corollary}
	Let $S$ be a yoked entire semiring and $M$ a cancellative faithful multiplication $S$-semimodule.
	Then the following statements
	hold:
	\begin{enumerate} 
		\item If $I$ and $J$ are ideals of $S$ such that $IM\subseteq JM$ then $I\subseteq J$.
		\item $M\neq IM$ for any proper ideal $I$ of $S$.
		\item M is finitely generated.
	\end{enumerate} 
\end{corollary}

An $S$-semimodule $M$ is called a cancellation semimodule if whenever $IM=JM$ for ideals $I$ and $J$ of $S$, then $I=J$\cite{F}. By Theorem \ref{3}, every $MC$ multiplication semimodule is a cancellation semimodule.

In \cite[Lemma 4.1]{E4} it is shown that faithful multiplication modules are torsion-free. Similarly, we have the following result.
\begin{theorem}
	\label{6}
	Suppose that $S$ is a semidomain. If $M$ is an $MC$ multiplication $S$-semimodule, then $M$ is a torsionfree $S$-semimodule.
\end{theorem}
\begin{proof}
	Suppose that there exist $0\neq t\in S$ and $m,m^{\prime}\in M$ such that $tm=tm^{\prime}$. Then $Sm=IM$ and $Sm^{\prime}=JM$ for some ideals $I, J$ of $S$. Thus $tIM=tJM$ since $tm=tm^{\prime}$. Since $M$ is an $MC$ multiplication $S$-semimodule, it is a cancellation semimodule. Thus $tI=tJ$. Let $x\in I$. Then $tx=tx^{\prime}$ for some $x^{\prime}\in J$. Since $S$ is a semidomain, $x=x^{\prime}$. Therefore $I\subseteq J$. Similarly $J\subseteq I$. Hence $I=J$ and $Sm=Sm^{\prime}$. Then there exists $s_1\in S$ such that $m=s_1m^{\prime}$. Thus $tm^{\prime}=tm=ts_1m^{\prime}$. Since $M$ is an $MC$ semimodule, $t=s_1t$. Since $S$ is a semidomain, $s_1=1$. Therefore $m=m^{\prime}$ and hence $M$ is torsionfree. 
\end{proof}
If $M$ is a finitely generated faithful multiplication module, then $M$ is a projective module \cite[Theorem 11]{S2}. Similarly, we have the following theorem: 
\begin{theorem}
	Let $M$ be an $MC$ multiplication semimodule. Then $M$ is a projective \linebreak$S$-semimodule.
\end{theorem}
\begin{proof}
	By Theorem \ref{3}, $\theta(M)=\sum_{i}^n(Sm_i:M)=S$. Thus for each $1\leq i\leq n$, there exist $r_i\in (Sm_i:M)$ and $s_i\in S$ such that $1=s_1r_1^2+\ldots+s_nr_n^2$. Define a map $\phi_i:M\rightarrow S$ by $\phi_i:m\mapsto s_ir_ia$ where $a$ is an element of $S$ such that $r_im=am_i$. Suppose that $am_i=bm_i$ for some $b\in S$. Since $M$ is an $MC$ semimodule, $a=b$ and therefore $\phi_i$ is a well defined $S$-homomorphism. Let $m\in M$. Then $m=1m= s_1r_1^2m+\cdots+s_nr_n^2m=\phi_1(m)m_1+\cdots+\phi_n(m)m_n$. By \cite[Theorem 3.4.12]{S1}, $M$ is a projective $S$-semimodule.
\end{proof}
By Lemma \ref{12}, we obtain the following result.
\begin{corollary}
	Let $S$ be a yoked entire semiring and $M$ a cancellative faithful multiplication $S$-semimodule. Then $M$ is a projective $S$-semimodule.
\end{corollary}
\begin{theorem}\cite[Lemma 3.6]{E4}
	Let $S$ be a semidomain and let $M$ be an $MC$ multiplication $S$-semimodule. Then there exists an invertible ideal $I$ of $S$ such that $M\cong I$.
\end{theorem}
\begin{proof}
	Suppose that $0\neq m\in M$. Then there exists an ideal $J$ of $S$ such that $Sm=JM$. Let $0\neq a\in J$. We can define an $S$-homomorphism $\phi:M\rightarrow Sm$ by $\phi:x\mapsto ax$. Let $x,x^{\prime}\in M$ such that $ax=ax^{\prime}$. By Theorem \ref{6}, $M$ is torsionfree and hence $x=x^{\prime}$. Therefore $\phi$ is injective and so $M\cong f(M)$. Now define an $S$-homomorphism $\phi^{\prime}:S\rightarrow Sm$ by $\phi^{\prime}(s)=sm$. Let $s,s^{\prime}\in S$ such that $sm=s^{\prime}m$. Since $M$ is an $MC$ semimodule, $s=s^{\prime}$. Therefore $\phi^{\prime}$ is injective. It is clear that $\phi^{\prime}$ is surjective. Therefore $S\cong Sm$. Hence $M$ is isomorphic to an ideal $I$ of $S$. Thus $I$
	is a multiplication ideal and hence an invertible ideal of $S$.
\end{proof}
\section{Cancellative multiplication semimodule}\label{sec:3} 
In this section, we investigate cancellative multiplication semimodules over some special semirings and restate some previous results.
From now on, let $S$ be a yoked semiring such that every maximal ideal of $S$ is subtractive and let $M$ be a cancellative \linebreak$S$-semimodule.

\begin{theorem}(See Theorems \ref{1} and \ref{17})
	\label{7}
	The $S$-semimodule $M$ is a multiplication \linebreak$S$-semimodule if and only if for every maximal ideal $\mathfrak{p}$ of $S$ either $M$ is $\mathfrak{p}$-cyclic or $M=\{m\in M\mid m=qm\ for\ some\ q\in \mathfrak{p}\}$.
\end{theorem}
\begin{proof}
	$(\rightarrow)$: Follows by Theorem \ref{1}.
	
	$(\leftarrow)$: Let $N$ be a subsemimodule of $M$ and $I=(N:M)$. Then $IM\subseteq N$. Let $x\in N$ and put $K=\{s\in S\mid sx\in IM\}$. If $K\neq S$, there exists a maximal ideal $\mathfrak{p}$ of $S$ such that $K\subseteq \mathfrak{p}$. If $M=\{m\in M\mid m=qm\ for\ some\ q\in \mathfrak{p}\}$, then there exists $q\in \mathfrak{p}$ such that $x=qx$. Since $S$ is a yoked semiring, there exists $t\in S$ such that $t+1=q$ or $q+t=1$. Suppose that $q+t=1$. Then $qx+tx=x$ and hence $tx=0$. Therefore $t\in K\subseteq \mathfrak{p}$ which is a contradiction. Now suppose that $t+1=q$. Then $tx+x=qx$ and hence $tx=0$. Therefore $t\in K\subseteq \mathfrak{p}$. But $\mathfrak{p}$ is a subtractive ideal of $S$, so $1\in \mathfrak{p}$ which is a contradiction. Therefore $M$ is $\mathfrak{p}$-cyclic. Thus there exist $m\in M$, $t\in S$ and  $q\in \mathfrak{p}$ such that $t+q=1$ and  $tM\subseteq Sm$. Therefore
	$tN$ is a subsemimodule of $Sm$. Hence $tN=Jm$ where $J$ is the ideal $\{s\in S\mid  sm\in tN\}$ of $S$. Then $tJM=JtM\subseteq Jm\subseteq N$ and hence $tJ \subseteq I$. Thus $t^2x\in t^2N=tJm\subseteq IM$. Therefore $t^2\in K\subseteq \mathfrak{p}$ which is a contradiction.
\end{proof}
\begin{lemma}\label{23}
	If $\mathfrak{p}$ is a maximal ideal of $S$, then $N=\{m\in M\mid m=qm\ for\ some\ q\in \mathfrak{p}\}$ is a subsemimodule of $M$.
	\begin{proof}
		Let $m_1,m_2\in N$. Then there exist $q_1,q_2\in \mathfrak{p}$
		such that $m_1=q_1m_1$ and $m_2=q_2m_2$. Since $S$ is a yoked semiring, there exits an element $r$ such that $q_1+q_2+r=q_1q_2$ or $q_1q_2+r=q_1+q_2$. Since $\mathfrak{p}$ is a subtractive ideal, $r\in \mathfrak{p}$. 
		
		Assume that $q_1q_2+r=q_1+q_2$. Then  $q_1q_2(m_1\!+m_2)+r(m_1\!+m_2)\!=\!(q_1+q_2)(m_1\!+m_2)$. Thus $q_1q_2m_1+q_1q_2m_2+r(m_1+m_2)=q_1m_1+q_2m_1+q_1m_2+q_2m_2$. Hence
		$q_2m_1+q_1m_2+r(m_1+m_2)=q_1m_1+q_2m_1+q_1m_2+q_2m_2$. Since $M$ is a cancellative $S$-semimodule,
		$r(m_1+m_2)=q_1m_1+q_2m_2$. Thus $r(m_1+m_2)=m_1+m_2$. Therefore $m_1+m_2\in N$. 
		
		Now assume that $q_1+q_2+r=q_1q_2$. Then $(q_1+q_2+r)(m_1+m_2)=q_1q_2(m_1+m_2)$. Hence  $q_1m_1+q_1m_2+q_2m_1+q_2m_2+r(m_1+m_2)=q_1q_2m_1+q_1q_2m_2$.
		Thus 
		$q_1m_1+q_1m_2+q_2m_1+q_2m_2+r(m_1+m_2)=q_2m_1+q_1m_2$. Since $M$ is a cancellative $S$-semimodule,
		$q_1m_1+q_2m_2+r(m_1+m_2)=0$ and hence $m_1+m_2+r(m_1+m_2)=(1+r)(m_1+m_2)=0$. Since $\mathfrak{p}$ is a subtractive ideal, $(1+r)\notin \mathfrak{p}$. Therefore $(1+r)+\mathfrak{p}=S$ since $\mathfrak{p}$ is a maximal ideal of $S$. Thus there exist $t\in \mathfrak{p}$ and $s\in S$ such that $s(1+r)+t=1$. Hence $s(1+r)(m_1+m_2)+t(m_1+m_2)=m_1+m_2$. Therefore $t(m_1+m_2)=m_1+m_2$ and so $m_1+m_2\in N$.
		
		Let $s\in S$ and $m\in N$. Then there exists $q\in \mathfrak{p}$ such that $m=qm$. Thus $sm=sqm$. Since $sq\in \mathfrak{p}$, $sm\in N$.
		Therefore $N$ is a subsemimodule of $M$. 
	\end{proof} 
\end{lemma}
Similar to \cite[Corollary 1.3]{E4}, we have the following theorem.
\begin{theorem}
	\label{8}
	Let $M=\sum_{\lambda\in \Lambda} Sm_{\lambda}$. Then $M$ is a multiplication semimodule if and only if there exist ideals $I_{\lambda}(\lambda\in \Lambda)$ of $S$ such that $Sm_{\lambda}=I_{\lambda}M$ for all $\lambda\in \Lambda$.
\end{theorem}
\begin{proof}
	$(\rightarrow)$: Obvious.
	
	$(\leftarrow)$: Assume that there exist ideals
	$I_{\lambda}(\lambda\in \Lambda)$ of $S$ such that $Sm_{\lambda}=I_{\lambda}M(\lambda\in \Lambda)$. Let $\mathfrak{p}$ be a maximal ideal of $S$ and $I_{\mu}\nsubseteq \mathfrak{p}$ for some $\mu\in \Lambda$. Then there exists $t\in I_{\mu}$ such that $t\notin \mathfrak{p}$. Thus $\mathfrak{p}+(t)=S$ and hence there exist $q\in \mathfrak{p}$ and $s\in S$ such that $1=q+st$. Then $tsM\subseteq I_{\mu}M=Sm_{\mu}$. Therefore $M$ is $\mathfrak{p}$-cyclic.
	Now suppose that $I_{\lambda}\subseteq \mathfrak{p}$ for all $\lambda\in \Lambda$. Then $Sm_{\lambda}\subseteq \mathfrak{p}M(\lambda\in \Lambda)$. This implies that $M=\mathfrak{p}M$. But for any $\lambda\in \Lambda$, $Sm_{\lambda}=I_{\lambda}M=I_{\lambda}\mathfrak{p}M=\mathfrak{p}m_{\lambda}$. Therefore 
	$m_{\lambda}\in \{m\in M\mid  m=qm\ for\ some\ q\in \mathfrak{p}\}$. Since by Lemma \ref{23}, $\{m\in M\mid m=qm\ for\ some\ q\in \mathfrak{p}\}$ is an $S$-semimodule, we conclude that $M=\{m\in M\mid m=qm\ for\ some\ q\in \mathfrak{p}\}$. By Theorem \ref{7}, $M$ is a multiplication semimodule.
\end{proof}
It follows from Theorem \ref{8} that if $S$ is a yoked semiring such that every maximal ideal of $S$ is subtractive, then any additively cancellative ideal $I$ generated by idempotents is a multiplication ideal.

The following is a generalization of \cite[Theorem 3.1]{E4}
\begin{theorem}
	Let $M$ be a faithful multiplication $S$-semimodule. 
	Then the following statements are equivalent:
	\begin{enumerate} 
		\item M is finitely generated.
		\item $M\neq \mathfrak{p}M$ for any maximal ideal $\mathfrak{p}$ of $S$.\label{10}
		\item If $I$ and $J$ are ideals of $S$ such that $IM\subseteq JM$ then $I\subseteq J$.
		\item For each subsemimodule $N$ of $M$ there exists a unique ideal $I$ of $S$ such that $N=IM$.
		\item $M\neq IM$ for any proper ideal $I$ of $S$.\label{11}
	\end{enumerate} 
\end{theorem}
\begin{proof}
	$(1) \rightarrow (2)$:
	Let $\mathfrak{p}$ be a maximal ideal of $S$ such that $M=\mathfrak{p}M$ and $M=Sm_1+\ldots+Sm_n$. Since $M$ is a multiplication $S$-semimodule, for each $1\leq i\leq n$, there exists $K_i\subseteq S$ such that $Sm_i=K_iM=K_i\mathfrak{p}M=\mathfrak{p}K_iM=\mathfrak{p}m_i$. Therefore $m_i=p_im_i$ for some $p_i\in \mathfrak{p}$. Since $S$ is a yoked semiring, there exists $t_i\in S$ such that $t_i+p_i=1$ or $1+t_i=p_i$. Suppose that $t_i+p_i=1$. Then $t_im_i+p_im_i=m_i$. Since $M$ is a cancellative $S$-semimodule, $t_im_i=0$. Now suppose that $1+t_i=p_i$. Then $m_i+t_im_i=p_im_i$ and hence $t_im_i=0$. Put $t=t_1\ldots t_n$. Then for all $i$, $tm_i=0$. Thus $tM=\{0\}$. Since $M$ is a faithful $S$-semimodule, $t=0\in \mathfrak{p}$. Since $\mathfrak{p}$ is a prime ideal, $t_i\in \mathfrak{p}$ for some $1\leq i\leq n$. If $t_i+p_i=1$, then $1\in \mathfrak{p}$ which is a contradiction. If $1+t_i=p_i$, then, since $\mathfrak{p}$ is a subtractive ideal of $S$, $1\in \mathfrak{p}$ which is a contradiction. Therefore $M\neq \mathfrak{p}M$.
	
	$(2) \rightarrow (3)$: 
	Let $I$ and $J$ be ideals
	of $S$ such that $IM\subseteq JM$. Let $a\in I$ and put $K=\{r\in S\mid ra\in J\}$. If
	$K\neq S$, then there exists a maximal ideal $\mathfrak{p}$ of $S$ such that $K\subseteq \mathfrak{p}$. By \ref{10}, $M\neq \mathfrak{p}M$. Thus $M$ is $\mathfrak{p}$-cyclic and hence there exist $m\in M$, $t\in S$ and  $q\in \mathfrak{p}$ such that $t+q=1$ and $tM\subseteq Sm$. Then $tam\in tJM=JtM\subseteq Jm$. Thus there exists $b\in J$ such that $tam=bm$. Since $S$ is a yoked semiring, there exists $c\in S$ such that $ta+c=b$ or $b+c=ta$. Suppose that $ta+c=b$. Then $t^2a+tc=tb$ and $tam+cm=bm$. Since $M$ is cancellative, $cm=0$. But $tcM\subseteq c(Sm)=\{0\}$. Since $M$ is a faithful semimodule, $tc=0$. Hence $t^2a=tb\in J$. Therefore $t^2\in K\subseteq \mathfrak{p}$ which is a contradiction. Thus $S=K$ and $a\in J$. Now suppose that $b+c=ta$. Then $tb+tc=t^2a$ and $bm+cm=tam$. Since $M$ is cancellative, $cm=0$. A similar argument shows that $a\in J$.
	
	$(3) \rightarrow (4)\rightarrow (5)$: Obvious.
	
	$(5) \rightarrow (1)$:
	It is similar to the proof of Theorem \ref{3}(3).
	
\end{proof}
Theorem \ref{18} can be restated as follows:
\begin{theorem}(cf. \cite[Proposition 3]{E2})
	Suppose that $\mathfrak{p}$ is a prime ideal and let $M$ be a faithful multiplication $S$-semimodule. Let $a\in S$ and $x\in M$ such that $ax\in \mathfrak{p}M$. Then $a\in \mathfrak{p}$ or $x\in \mathfrak{p}M$.
\end{theorem}
\begin{proof}
	Let $a\notin \mathfrak{p}$ and $K=\{s\in S\mid sx\in \mathfrak{p}M \}$. Assume that $K\neq S$.  Then there exists a maximal ideal $\mathfrak{q}$ of $S$ such that $K\subseteq \mathfrak{q}$. Let $M=\mathfrak{q}M$. Then there exists an ideal $I$ of $S$ such that $Sx=IM$. Thus $Sx=I\mathfrak{p}M=\mathfrak{p}IM=\mathfrak{p}x$. Hence $x=qx$ for some $q\in \mathfrak{p}$.
	Since $S$ is a yoked semiring, there exists $t\in S$ such that $t+1=q$ or $q+t=1$. If $q+t=1$, then $qx+tx=x$ and hence $tx=0$. Therefore $t\in K\subseteq \mathfrak{p}$ which is a contradiction. Now suppose that $t+1=q$. Then $tx+x=qx$ and hence $tx=0$. Thus $t\in K\subseteq \mathfrak{p}$. But $\mathfrak{p}$ is a subtractive ideal of $S$, so $1\in \mathfrak{p}$ which is a contradiction.
	Hence $M\neq \mathfrak{q}M$. Thus by Theorem \ref{1}, $M$ is $\mathfrak{q}$-cyclic. Therefore there exist $m\in M$, $t\in S$ and $q\in \mathfrak{q}$ such that $t+q=1$ and $tM\subseteq Sm$. Thus $tx=sm$ for some $s\in S$. Since $t\mathfrak{p}M\subseteq \mathfrak{p}m$, $tax\in t\mathfrak{p}M\subseteq \mathfrak{p}m$. Hence $tax=p_1m$ for some $p_1\in \mathfrak{p}$. Then $asm=p_1m$. Since $S$ is a yoked semiring, there exists $c\in S$ such that $as+c=p_1$ or $c+p_1=as$. Suppose that $as+c=p_1$. Then $asm+cm=p_1m$. Since $M$ is cancellative, $cm=0$. Then $tcM\subseteq c(Sm)=\{0\}$. Since $M$ is a faithful semimodule, $tc=0$. Hence $ast=p_1t\in \mathfrak{p}$ and so $s\in \mathfrak{p}$ since $\mathfrak{p}$ is a prime ideal. Then $tx=sm\in \mathfrak{p}M$ and hence $t\in K\subseteq \mathfrak{q}$ which is a contradiction. Thus $K=S$. Therefore $x\in \mathfrak{p}M$. Now suppose that $c+p_1=as$. A similar argument shows that $x\in \mathfrak{p}M$.
\end{proof}

\end{document}